\documentclass[a4paper]{amsart}

\usepackage{amssymb,latexsym}
\usepackage{amsmath}
\usepackage{amsfonts}
\usepackage{amsthm}
\usepackage{amssymb, comment}
\usepackage{hyperref}

\theoremstyle{plain}
\newtheorem{theorem}{Theorem}

\def\Q{\mathbb{Q}}

\theoremstyle{definition}
\newtheorem{definition}{Definition}

\title[Seven squares from three numbers]
{Seven squares from three numbers}
\begin{document}

\date{}


\author[A. Dujella]{Andrej Dujella}
\address{
Department of Mathematics\\
Faculty of Science\\
University of Zagreb\\
Bijeni{\v c}ka cesta 30, 10000 Zagreb, Croatia
}
\email[A. Dujella]{duje@math.hr}

\author[M. Kazalicki]{Matija Kazalicki}
\address{
Department of Mathematics\\
Faculty of Science\\
University of Zagreb\\
Bijeni{\v c}ka cesta 30, 10000 Zagreb, Croatia
}
\email[M. Kazalicki]{matija.kazalicki@math.hr}

\author[V. Petri\v{c}evi\'c]{Vinko Petri\v{c}evi\'c}
\address{
Faculty of Electrical Engineering, Computer Science and Information Technology Osijek\\
Josip Juraj Strossmayer University of Osijek\\
Kneza Trpimira 2B, HR-31000 Osijek, Croatia
}
\email[V. Petri\v{c}evi\'c]{vinko.petricevic@ferit.hr}

\begin{abstract}
	We study triples \(\{a,b,c\}\) of distinct nonzero rational numbers such that
	\[
	a+1,\quad b+1,\quad c+1,\quad ab+1,\quad ac+1,\quad bc+1,\quad abc+1
	\]
	are all perfect squares. We prove that there exist infinitely many such triples. In contrast, we show that no triple of positive integers has this property.
\end{abstract}

\subjclass[2020]{Primary 11D09; Secondary 11G05}
\keywords{Diophantine triples; rational Diophantine tuples; regular Diophantine tuples; elliptic
	curves}

\maketitle

\section{Introduction}
A Diophantine \(m\)-tuple is a set of \(m\) distinct positive integers with the property that the product of any two distinct elements, increased by \(1\), is a perfect square. The first known example is Fermat's quadruple
\[
\{1,3,8,120\}.
\]
More generally, one may consider rational Diophantine \(m\)-tuples, defined in the same way for sets of nonzero rational numbers. An example already known to Diophantus is
\[
\left\{\frac{1}{16},\,\frac{33}{16},\,\frac{17}{4},\,\frac{105}{16}\right\}.
\]
Euler proved that there exist infinitely many rational Diophantine quintuples \cite{Hea}; in particular, he extended Fermat's quadruple to
\[
\left\{1,3,8,120,\frac{777480}{8288641}\right\}.
\]
Stoll later proved that this extension is unique \cite{Stoll}.

In the integral case, Baker and Davenport proved that if \(\{1,3,8,d\}\) is a Diophantine quadruple, then \(d=120\) \cite{B-D}. This led to the conjecture that there are no Diophantine quintuples in integers, which was recently proved by He, Togb\'e, and Ziegler \cite{HTZ}; see also \cite{duje-crelle}.

For rational Diophantine tuples, the maximal size is still unknown. Gibbs found the first rational Diophantine sextuple \cite{Gibbs1},
\[
\left\{\frac{11}{192},\,\frac{35}{192},\,\frac{155}{27},\,\frac{512}{27},\,\frac{1235}{48},\,\frac{180873}{16}\right\}.
\]
Since then, several constructions of infinitely many rational Diophantine sextuples have been found. In particular, Dujella, Kazalicki, Miki\'c, and Szikszai \cite{DKMS}, Dujella and Kazalicki \cite{Duje-Matija}, and Dujella, Kazalicki, and Petri\v{c}evi\'c \cite{DKP-sext,DKP-reg} constructed different infinite families of rational Diophantine sextuples, with \cite{Duje-Matija} also building on ideas of Piezas \cite{P}. No example of a rational Diophantine septuple is currently known. Lang's conjecture implies that the size of a rational Diophantine tuple is bounded by an absolute constant; see the introduction of \cite{DKMS}. For more background, see \cite{Duje-notices,duje-book}.

A closely related problem was considered by Dujella and Szalay \cite{DujellaSzalay}. They studied triples of positive integers \(a,b,c>1\) such that
\[
ab+1,\quad ac+1,\quad bc+1,\quad abc+1
\]
are all perfect squares, and proved that there exist infinitely many such triples.

This problem, attributed to John Gowland, appeared in the sci.math newsgroup in 1998 and was later featured in Brown's book \cite{Brown}. Their result also resolves a conjecture from Kenta Takahashi's 2010 thesis, which states that there exist infinitely many integer quadruples with the property that the product of any three of their elements, increased by one, is a perfect square.

In this paper, we study triples \(\{a,b,c\}\) of distinct nonzero rational numbers such that
\begin{equation}\label{eq:maincondition}
	a+1,\quad b+1,\quad c+1,\quad ab+1,\quad ac+1,\quad bc+1,\quad abc+1
\end{equation}
are all perfect squares. We call such triples $\textit{exotic}$ Diophantine triples. Equivalently, \(\{1,a,b,c\}\) is a rational Diophantine quadruple for which \(abc+1\) is also a perfect square. Our main results are the following.

\begin{theorem}\label{thm:1}
	There exist infinitely many exotic rational Diophantine triples.
\end{theorem}

For integers, the situation is different.

\begin{theorem}\label{thm:2}
There are no exotic Diophantine triples in positive integers.
\end{theorem}

\section{Experiment and proof of Theorem \ref{thm:1}}

We started our investigation with a numerical search of exotic Diophantine triples $\{a,b,c\}$.
The following two examples turned out to be particularly relevant: $$\{8, 312/529, 495/529\}, \{312/529, -152880/165649, -78374557/87628321\}.$$

To understand these examples, we introduce the concept of regularity (see \cite{DKP-reg,DP2}).

\begin{definition}
	The triple $(a,b,c)\in \Q^3$ is called {\it regular} if $r_3(a,b,c)=0$ where $$r_3(a,b,c)=(a+b-c)^2-4(ab+1).$$
	Similarly, the quadruple $(a,b,c,d)\in \Q^4$ is called regular if $r_4(a,b,c,d)=0$ where
	$$r_4(a,b,c,d)=(a+b-c-d)^2-4(ab+1)(cd+1).$$ 
	Note that polynomials $r_3$ and $r_4$ are symmetric.
\end{definition}

Our first observation was that for almost half of the exotic triples we found (for $24$ out of $56$), the quadruple \(\{1,a,b,c\}\) is a regular Diophantine quadruple, so we decided to focus on this case.

The next difficulty was to find an elegant algebraic formulation of the condition that \(abc+1\) is a perfect square. The key observation is that, for every exotic triple, the set \(\{1,ab,c\}\) is a Diophantine triple. Moreover, if this triple is regular, as in the two examples above, then \(r_3(1,ab,c)=0\), giving an algebraic condition that forces  \(abc+1\) to be a perfect square.

Conversely, the symmetry of the polynomial \(r_3\) implies that whenever a triple \(\{a,b,c\}\) satisfies \(r_3(1,ab,c)=0\), the quantities \(ab+1\), \(abc+1\), and \(c+1\) are automatically perfect squares. It is also easy to see from the symmetries of the polynomial \(r_4\) that if the same triple of distinct nonzero rationals satisfies \(r_4(1,a,b,c)=0\), then the only conditions missing for \(\{a,b,c\}\) to be an exotic Diophantine triple are that \(a+1\) and \(b+1\) are perfect squares. This suggests setting
\[
a=r^2-1,\qquad b=s^2-1,
\]
and hence studying the system of equations
\[
r_3(1,(r^2-1)(s^2-1),c)=0,\qquad r_4(1,r^2-1,s^2-1,c)=0.
\]

The difference of these two polynomials factors as
\[
(rs-1)(rs+1)\bigl(r^2s^2-2s^2-2r^2+5+2c\bigr)=0.
\]

If \((rs-1)(rs+1)=0\), then substituting \(r=\pm 1/s\) into
\[
r_4(1,r^2-1,s^2-1,c)=0
\]
gives the genus one curve
\[
s^8 + 2s^6c - 2s^6 + s^4c^2 - 6s^4c - s^4 + 2s^2c - 2s^2 + 1=0,
\]
which is birationally equivalent to the elliptic curve
\[
E_1:\ y^2=x^3-7x-6.
\]
Since the Mordell--Weil rank of \(E_1\) over \(\Q\) is zero, this case does not yield infinitely many solutions.

In the second case, substituting
\[
c=\frac{-r^2s^2+2s^2+2r^2-5}{2}
\]
into
\[
r_4(1,r^2-1,s^2-1,c)=0
\]
leads to the two genus one curves
\[
r^2s^2-\frac{4}{3}r^2-\frac{2}{3}rs-\frac{4}{3}s^2+\frac{7}{3}=0
\]
and
\[
r^2s^2-\frac{4}{3}r^2+\frac{2}{3}rs-\frac{4}{3}s^2+\frac{7}{3}=0.
\]
These curves are birationally equivalent to the elliptic curve
\[
E_2:\ y^2=x^3-111x+450,
\]
whose Mordell--Weil group over \(\Q\) is generated by the point \([6,0]\) of order two and the point \([3,-12]\) of infinite order.

Thus, every point $(x,y)$ on $E_2$ gives an exotic Diophantine triple $(r^2-1,s^2-1,c)$ (provided the resulting
triple is nondegenerate) where

\[
u=\frac{2y-6x+42}{x^2-6x-3},
\]
\[
s=\frac{x^2-12x+39+2y}{x^2-6x-3},
\]
\[
r=\frac{(x-3)u^2+7u-1}{3u^2+6u-1},
\]
and
\[
c=\frac{-r^2s^2+2s^2+2r^2-5}{2}.
\]
Note that points $P$ and $P+[6,0]$ on $E_2(\Q)$ give the same exotic triple.

The two examples at the beginning of this section correspond to the points \([7,4]=[6,0]-[3,-12]\) and \([-6,-30]=[6,0]-2\cdot[3,-12]\) on \(E_2(\Q)\).
For more examples, one can check that the points
\[
\left(\frac{223}{9},-\frac{3068}{27}\right),\qquad
\left(\frac{13206}{2401},\frac{285060}{117649}\right),\qquad
\left(\frac{2746063}{444889},\frac{45138812}{296740963}\right)
\]
on \(E_2(\Q)\) correspond to the triples
\[
\left\{
\frac{724255280}{736742449},
-\frac{152880}{165649},
-\frac{63009087694401}{122040649934401}
\right\},
\]
\[
\left\{
\frac{24490915482072}{12448992625969},
\frac{724255280}{736742449},
-\frac{4510665894525110607837}{9171701314839342058081}
\right\},
\]
and
\[
\left\{
-\frac{2539564321528123032}{5054545907282329441},
\frac{24490915482072}{12448992625969},
\frac{14261842404349331345950974819695}{62924004727379507987985949853329}
\right\},
\]
respectively.

\section{Proof of Theorem \ref{thm:2}}
\begin{proof}[Proof of Theorem \ref{thm:2}]
Assume for contradiction that such integers exist. Since their increments by $1$ are perfect squares and they are nonzero, we may order them without loss of generality as $3 \leq a < b < c$ (the lower bound $3$ follows from $a+1$ being a square). Let $ab+1 = r^2$, $ac+1 = s^2$, $bc+1 = t^2$, $c+1 = z^2$, and $abc+1 = u^2$ for positive integers $r, s, t, z, u$. The set $\{1, a, b, c\}$ forms an integer Diophantine quadruple, so Lemma 14 from \cite{duje_abs} implies that $c >  4ab$.

Define the integer $M := 2r(zu - st)$. We first establish the identity

$$ (zu)^2 - (st)^2 = (c+1)(abc+1) - (ac+1)(bc+1) = c(a-1)(b-1), $$
which yields $zu - st = \frac{c(a-1)(b-1)}{zu + st}$, and consequently $M = \frac{2rc(a-1)(b-1)}{zu+st}$.

Next, we bound $zu+st$. For the lower bound, $zu+st > \sqrt{c \cdot abc} + \sqrt{ac \cdot bc} = 2c\sqrt{ab}$. For the upper bound, observe that since $c > 4ab \ge 12$, we have $(rc)^2 = abc^2+c^2 > \max((zu)^2, (st)^2)$, which implies $zu < rc$ and $st < rc$, yielding $zu+st < 2rc$.

Substituting these bounds into the expression for $M$, we obtain

$$ (a-1)(b-1) < M < \frac{2rc(a-1)(b-1)}{2c\sqrt{ab}} = \frac{r}{\sqrt{ab}}(a-1)(b-1). $$

Since $\frac{r}{\sqrt{ab}} = \sqrt{1 + \frac{1}{ab}} < 1 + \frac{1}{2ab}$, we find

$$ (a-1)(b-1) < M < \left(1 + \frac{1}{2ab}\right)(a-1)(b-1) < (a-1)(b-1) + \frac{1}{2}. $$

This places the integer $M$ strictly between $(a-1)(b-1)$ and $(a-1)(b-1) + 1$, which is a contradiction. Therefore, no such integers $a, b, c$ can exist.
\end{proof}

{\bf Acknowledgements.}
The authors wish to thank Tomislav Pejković for suggesting this problem. They also acknowledge the assistance of the Cantab Pi AI Math Research Assistant, which contributed to finding the proof of Theorem \ref{thm:2}.

The authors were supported by the Croatian Science Foundation under the project no.\ IP-2022-10-5008 (TEBAG). The authors acknowledge support from the project “Implementation of cutting-edge research and its application as part of the Scientific Center of Excellence for Quantum and Complex Systems, and Representations of Lie Algebras”, Grant No.\ PK.1.1.10.0004, co-financed by the European Union through the European Regional Development Fund -- Competitiveness and Cohesion Programme 2021-2027. This research was funded by the European union: NextGenerationEU (Digital Innovations in Teaching Foundational Courses for Engineers of the Future 581-UNIOS-39), and NextGenerationEU through the National Recovery and Resilience Plan 2021-2026. Institutional grant of University of Zagreb Faculty of Science (IK IA 1.1.3. Impact4Math).

\end{document}